\documentclass[11pt,a4paper]{amsart}

\usepackage[utf8]{inputenc}
\usepackage{amsthm,amsfonts,amsmath,amssymb,graphicx}
\usepackage[colorlinks=true,citecolor=black,linkcolor=black,urlcolor=blue]{hyperref}
\usepackage[numbers,sort&compress]{natbib}
\allowdisplaybreaks

\renewcommand{\geq}{\geqslant}
\renewcommand{\leq}{\leqslant}

\theoremstyle{plain}
\newtheorem{theorem}{Theorem}
\newtheorem{proposition}[theorem]{Proposition}
\newtheorem{conjecture}[theorem]{Conjecture}
\newtheorem{lemma}[theorem]{Lemma}
\newtheorem{corollary}[theorem]{Corollary}
\newtheorem{observation}[theorem]{Observation}

\newcommand{\floor}[1]{\ensuremath{\protect\left\lfloor#1\right\rfloor}}
\newcommand{\ceil}[1]{\ensuremath{\protect\left\lceil#1\right\rceil}}
\DeclareMathOperator{\N}{\mathbb{N}}

\newcommand{\bc}{\begin{center}}
\newcommand{\ec}{\end{center}}

\begin{document}
%%%%%%%%%%%%%%%%

\title{Bichromatic lines in the plane}
\date{\today}

\author[]{Michael~S.~Payne}
\address{
\newline Computer Science Department 
\newline Universit\'e Libre de Bruxelles
\newline Belgium}
\email{mpayne@ulb.ac.be}

%\thanks{Research supported by \ldots}

\begin{abstract} Given a set of red and blue points in the plane, a bichromatic line is a line containing at least one red and one blue point.
We prove the following conjecture of Kleitman and Pinchasi (unpublished, 2003). 
Let $P$ be a set of $n$ red, and $n$ or $n-1$ blue points in the plane. If neither colour class is collinear, then $P$ determines at least $|P|-1$ bichromatic lines.
In fact we are able to achieve the same conclusion under the weaker assumption that $P$ is not collinear or a near-pencil.
\end{abstract}

\maketitle

%%%%%%%%%%%%%%%%%%%%%%%%%%%%%%%%%%%%%%%%%%%%%%%%%%%%% Intro stuff
\section{Introduction}

In this paper we consider sets of red and blue points in the Euclidean plane.
If $P$ is such a set, a line containing two or more points of $P$ is said to be \emph{determined} by $P$.
A line determined by at least one red and one blue point is called \emph{bichromatic}.

In 2003, Kleitman and Pinchasi~\cite{kleitpinch} studied lower bounds on the number of bichromatic lines under the assumption that neither colour class is collinear. They made the following conjecture.

\begin{conjecture}[Kleitman--Pinchasi Conjecture]\label{KPconj}
Let $P$ be a set of $n$ red, and $n$ or $n-1$ blue points in the plane. If neither colour class is collinear, then $P$ determines at least $|P|-1$ bichromatic lines.
\end{conjecture}

This conjecture is tight for the arrangement of $n-1$ red and $n-1$ blue points on a line, along with one red and one blue point off the line, and collinear with some point on the line.

%One motivation for this conjecture is that it would imply Theorem~\ref{dBErd} of de Bruijn and Erd\H os~\cite{BE48}, which states that every non-collinear set of $n$ points in the plane determines at least $n$ lines.

In 1948, de Bruijn and Erd\H os~\cite{BE48} proved that every non-collinear set of $n$ points in the plane determines at least $n$ lines.
In fact, they proved this result in a more general combinatorial setting. 

\begin{theorem}[de Bruijn and Erd\H os]\label{dBErd}
Let $S$ be a set of cardinality $n$ and $\{ S_1, \ldots, S_k \}$ a collection of subsets of $S$ such that each pair of elements in $S$ is contained in exactly one $S_i$. Then either $S=S_i$ for some $i$, or $k\geq n$.
\end{theorem}

As noted by de Bruijn and Erd\H os, the special case where $S$ is a set of points in the plane and the $S_i$ are the collinear subsets of $S$ is easier to prove than the general theorem.
It follows by induction from the well-known Sylvester-Gallai Theorem (actually first proven by Melchior~\cite{Melchior-DM41}), which says that every finite non-collinear set of points in the plane determines a line with just two points.
%In the case of bichromatic lines, Kleitman and Pinchasi~\cite{kleitpinch} conjectured that if $P$ is a set of $n$ red, and $n$ or $n-1$ blue points in the plane and neither colour class is collinear, then $P$ determines at least $|P|-1$ bichromatic lines.
As motivation, Kleitman and Pinchasi note that together with the following theorem of Motzkin~\cite{motzkinnonmixed}, Conjecture~\ref{KPconj} would imply the plane case of Theorem~\ref{dBErd}.

\begin{theorem}[Motzkin]
Every non-collinear set of red and blue points in the plane determines a monochromatic line. 
\end{theorem}

%%%%%%%%%%%%% Leave out the following to give us time to solve it

%We prove a slight strengthening of Conjecture~\ref{KPconj}, and also show that, unlike Theorem~\ref{dBErd}, it is not true in a purely combinatorial setting.

%%A similar combinatorial version of the problem has been studied by Meshulam~\cite{meshulam}. %, though only under the assumption that not all points are collinear.

%%\begin{theorem}[Meshulam]\label{meshu}
%%Let $X_1,\ldots X_c$ be disjoint sets of cardinality $n$ (these are colour classes), let $S= \bigcup_i X_i$ and let $\{ S_1, \ldots, S_k \}$ be a collection of subsets of $S$ such that each pair of elements in $S$ is contained in exactly one $S_i$ (these are `lines'). Then either $S= S_i$ for some $i$ or $ |\{i : \forall j \enspace S_i  \not\subset X_j  \}| \geq (c-1)n $ (this counts non-monochromatic `lines').
%%\end{theorem}

%%In the bichromatic case with $c=2$ we have at least $n$ bichromatic lines, roughly half the number conjectured by Kleitman and Pinchasi under the stronger assumption that no colour class is collinear.
%%It is an interesting question whether the lower bound of Theorem~\ref{meshu} can be improved under this assumption.

%%%%%%%%%%%%%%%%%%%%%%%%%%%%%%%%

%%%%%%%%%%%%%%%%%%%%%%%%%%%%%%%%%%%%%%%%%%%%%%%%%%%%%%%%% Chapter stuff

Kleitman and Pinchasi~\cite{kleitpinch} came very close to proving Conjecture~\ref{KPconj}, establishing the following theorem.
\begin{theorem}[Kleitman and Pinchasi]\label{KPthm}
Let $P$ be a set of $n$ red, and $n$ or $n-1$ blue points in the plane. If neither colour class is collinear, then $P$ determines at least $|P|-3$ bichromatic lines.
\end{theorem}
Purdy and Smith~\cite{purdysmith} proved Conjecture~\ref{KPconj} for $n\geq 79$.
%using their Theorem~\ref{pursmith} and a result of Kelly and Moser~\cite{KellyMoser58}. 
%In this section we improve on the methods of Kleitman and Pinchasi and show, firstly, that Conjecture~\ref{KPconj} is true for $n\geq 10$, and secondly, that for all $n$ the number of bichromatic lines is at least $|P|-2$.
We will establish the following strengthening of Conjecture~\ref{KPconj}.
\begin{theorem}\label{mainthm}
Let $P$ be a set of $n$ red, and $n$ or $n-1$ blue points in the plane. If $P$ is not collinear or a near-pencil,
%contains at most $|P|-2$ collinear points, 
then $P$ determines at least $|P|-1$ bichromatic lines.
\end{theorem}

\section{Preliminaries}

We begin with a few useful observations.

\begin{lemma}\label{imp21}
Suppose $P$ is a set of $n$ red and $n$ (or $n-1$) blue points, and suppose there is a line $L$ with $r$ red and  $b$ blue points. Let $r'= \min \{ n-r,b \}$ and $b'= \min \{ n-b,r \}$ (or $b'= \min \{ n-1-b,r \}$). Then the number of bichromatic lines is at least $$ \sum_{i=0}^{r'-1} b-i + \sum_{i=0}^{b'-1} r -i = br' - \frac{1}{2}r'(r'-1) + rb'-\frac{1}{2}b'(b'-1) \enspace .$$
Moreover, if $b+r < n$, then $r'=b$, $b'=r$ and the number of bichromatic lines is at least $(b^2 + b + r^2 +r)/2 $.
If $L$ is itself bichromatic we may add one more to these totals. 
\end{lemma}
\begin{proof}
The bichromatic lines with a red point on $L$ are distinct from those with a blue point. To count those with a red point, take any $b'$ blue points not on $L$. Order these blue points $p_1, p_2, p_3,\ldots, p_{b'}$. There are $r$ lines from $p_1$ to the red points on $L$. For $p_2$ there are also $r$ such lines, but $p_1$ may lie on one of them (but not more). So there are $r-1$ lines that were not yet counted. Similarly, for $p_3$ there are at least $r-2$ lines that are not counted previously, and for $p_i$ there are $r-i+1$.
\end{proof}

\begin{proposition}\label{propcoll} There is no counterexample to Theorem~\ref{mainthm} with one colour class collinear.
\end{proposition}

\begin{proof}
%Let us consider the case in which 
Suppose one colour class lies on a line $L$. 
If red is collinear, then using a similar idea to the proof of Lemma~\ref{imp21} we see that there are at least $n + (n-1)$ bichromatic lines, unless there is only one blue point not on $L$. In that case $P$ is a near-pencil.
If blue is collinear and there are $n$ blue points, the same argument applies.
Now suppose blue is collinear and there are $n-1$ blue points.
If $L$ is bichromatic, we get $(n-1) + (n-2) +1 = |P|-1$ bichromatic lines.
If $L$ is monochromatic, we have at least $(n-1)+(n-2) +(n-3)$ bichromatic lines, which suffices as long as $n\geq 4$.
Finally, if $n=3$ and $L$ is monochromatic, each of the two blue points on $L$ lies in two bichromatic lines, otherwise there would be four collinear points and $P$ would be a near-pencil. 
%
%This implies the following proposition.
\end{proof}

It is simple to check that this implies the following. 

\begin{corollary}\label{easycoro} There is no counterexample to Theorem~\ref{mainthm} with $|P|-2$ collinear points.
\end{corollary}

Using these observations we can establish the following strengthening of Claim 2.1 in \cite{kleitpinch}. %We can prove using the observations.% Proposition~\ref{} and Corollary~\ref{}.

\begin{lemma}\label{KPlem}
There is no counterexample to Theorem~\ref{mainthm} with $n$ collinear points.
%Let $P$ be a set of $n$ red, and $n$ or $n-1$ blue points in the plane, with neither colour class collinear. If $P$ determines a line with $n$ points or more, then there are at least $|P|-1$ bichromatic lines.
\end{lemma}
\begin{proof}%[Proof of Lemma~\ref{KPlem} using Lemma~\ref{imp21}]
Suppose there is a line $L$ with $n$ or more points of $P$.
Proposition~\ref{propcoll} implies that $L$ is bichromatic and that there is at least one red and one blue point not on $L$.
Suppose there are at least two of each colour not on $L$.
Then there are at least two bichromatic lines through all except two of the points on $L$.
Along with $L$ this yields $2n -2 +1 \geq |P|-1$ bichromatic lines.

Corollary~\ref{easycoro} says that there are at least three points not on $L$.
So now suppose there is only one point $p$ of some colour not on $L$, and hence at least two  of the other colour, say $q_1$ and $q_2$.
If $p$ is red, there are $n-1$ red points and at least one blue point on $L$.
There are at least $(n-1)+(n-2)$ bichromatic lines through the red points on $L$ and $\{q_1,q_2\}$, one bichromatic line through $p$ and the blue point on $L$, and $L$ itself, giving $2n-1$.

Finally, if $p$ is blue and there are $n-1$ blue points in total, then there are $n-2$ blue points and at least two red points on $L$.
This gives at least $(n-2)+(n-3)$ bichromatic lines through the blue points on $L$ and $\{q_1,q_2\}$, two bichromatic lines through $p$ and the red points on $L$, and $L$ itself, giving $2n-2=|P|-1$.
\end{proof}

%\begin{proof} By Proposition~\ref{propcoll} the two points not collinear with the rest of $P$ could not be the same colour.
%\end{proof}

%Thus from now on we assume that no colour class is collinear, which implies that $P$ is not a near-pencil.

\section{Large minimal counterexamples}

Kleitman and Pinchasi use proof by induction on the size of $P$ to establish Theorem~\ref{KPthm}.
They establish an inductive step that works for $n\geq 20$ for both Theorem~\ref{KPthm} and Conjecture~\ref{KPconj}. 
In this section we reproduce their argument for the sake of completeness, with a few simplifications.
We will also recast their argument in terms of a search for a minimal counterexample.

Suppose that $P$ is a smallest counterexample to Theorem~\ref{mainthm}, so removing a point from $P$ cannot yield another counterexample. 
%By Proposition~\ref{propcoll}, $P$ has neither colour class collinear.
%
Let $s_{i,j}$ be the number of lines determined by $P$ with exactly $i$ red points and $j$ blue points, where we always assume $i+j \geq 2$.

\begin{lemma}\label{1redlem} We may assume that $s_{1,j}=0$ for all $j$. In particular $s_{1,1}=0$, so every line determined by just two points is monochromatic. Moreover, by symmetry, $s_{i,1}=0$ for all $i$ in the case of $n$ blue points.
\end{lemma}
\begin{proof} If $s_{1,j}\geq 1$, removing the red point from such a line would yield either a near-pencil or a smaller counterexample.
In the first case, $P$ had all but two points on a line, contradicting Corollary~\ref{easycoro}.
%and since $P$ has neither class collinear, it has a point of either colour not on $L$. It can easily be seen that $P$ was not a counterexample, a contradiction.
%In the second case $P$ was not minimal, also a contradiction.  
\end{proof}

Let $S$ be the number of unordered pairs of points in $P$ with the same colour, and let $D$ be the number of unordered pairs with different colours. If there are $n$ blue points then $S- D = 2\binom{n}{2} - n^2 =-n $. If there are $n-1$ blue points then $S-D= \binom{n}{2} + \binom{n-1}{2} - n(n-1) = 1-n $. 
Thus $S-D \leq 1-n$.

Clearly $D = \sum_{i,j\geq1} ijs_{i,j}$. Ignoring the contribution of monochromatic lines with three or more points, we also have 
\begin{equation}\label{eqn1}
S \geq s_{2,0} + s_{0,2} + \sum_{i,j\geq1} \left(\binom{i}{2}+\binom{j}{2} \right)s_{i,j} \enspace.
\end{equation}

We use a classical inequality due to Melchior~\cite{Melchior-DM41}. 
%The proof uses Euler's formula applied to the projective dual configuration. 
% Melchior's Inequality was later rediscovered by Kelly and Moser~\cite{KellyMoser58}.
Let $t_i$ be the number of lines containing $i$ points in $P$.

\begin{theorem}[Melchior's Inequality]\label{melchior}
Let $P$ be a non-collinear set of points. Then $$ t_2 \geq 3+ \sum_{i\geq 3}(i-3)t_i \enspace .$$
\end{theorem}

Since $t_2=s_{2,0} + s_{0,2}$ by Lemma~\ref{1redlem}, combining Theorem~\ref{melchior} with (\ref{eqn1}) we get 
$$S-D \geq 3+\sum_{ i,j \geq 1}(i+j-3)s_{i,j}  + \sum_{i,j\geq1} \left(\binom{i}{2}+\binom{j}{2} -ij \right)s_{i,j} \enspace. $$
Using $S-D \leq 1-n$ this gives the following lower bound on (twice) the number of bichromatic lines.
\begin{equation}\label{eqn2}
2\sum_{i,j\geq 1} s_{i,j} \geq n+2+ \sum_{i,j\geq1} \left(\frac{1}{2}\left((i-j)^2 +i+j\right) -1 \right)s_{i,j} \enspace.
\end{equation}
Note that coefficients of the $s_{i,j}$ on the right hand side of (\ref{eqn2}) are all positive because we don't allow $i=j=1$. We wish to minimise the right hand side subject to the constraint
\begin{equation}\label{eqn3}
 D=\sum_{i,j\geq 1} ij s_{i,j} \geq n(n-1) \enspace.
\end{equation}

The miminum is acheived when the only non-zero $s_{i,j}$ is the one for which the ratio 
\begin{equation}\label{ratio}
 \frac{\frac{1}{2}\left((i-j)^2 +i+j\right) -1}{ij}
\end{equation}
 of the coefficients in (\ref{eqn2}) and (\ref{eqn3}) is minimised. This is because this $s_{i,j}$ simultaneously contributes the least to the right hand side of (\ref{eqn2}) and the most to the left hand side of (\ref{eqn3}).
Clearly this minimum is acheived when $i=j$ since this minimises both the difference and the sum relative to the product (this is the arithmetic-geometric mean inequality).
So (\ref{ratio}) becomes $(i-1)/i^2$, which decreases as $i$ grows larger for $i\geq2$.

Now by Lemmas~\ref{imp21} and~\ref{KPlem}, we have that $\frac{1}{2}(i^2 + j^2 + i +j)\leq 2n-2$. This restricts $(i,j)$ to lie within a circle centred at $(-\frac{1}{2},-\frac{1}{2})$.
The minimum of (\ref{ratio}) still occurs on the line $i=j$ for this domain. To see this note that the curves on which (\ref{ratio}) is constant are hyperbolas that are symmetric about $i=j$ and tangent to circles centred on $i=j$.
Thus if $k$ is the maximum integer such that $k^2+k \leq 2n-2$, then $s_{k,k}$ is the non-zero variable that minimises the right hand side of (\ref{eqn2}).
Therefore we may set $k = \floor{\sqrt{2n}}$.

The constraint~(\ref{eqn3}) implies that $k^2 s_{k,k} \geq n(n-1)$, which implies $s_{k,k} \geq n(n-1)/k^2$.
Since $P$ is a counterexample, $2\sum_{i,j\geq 1} s_{i,j} \leq 4n-4$.
Combining all this with (\ref{eqn2}) gives
\begin{equation}\label{eqn5}
4n-4 \geq n+2+ (\floor{\sqrt{2n}} -1 )\frac{n(n-1)}{(\floor{\sqrt{2n}})^2} \enspace.
\end{equation}
The right hand side of (\ref{eqn5}) grows as $\Omega(n^{3/2})$ and the left hand side linearly, so it must be false for large $n$. One can check that it is false for all $n\geq21$.
Therefore any minimal counterexamples to Theorem~\ref{mainthm} must occur with $n \leq 20$.

%\clearpage

%%%%%%%%%%%%%%%%%%%%%%%%%%%%%%%%%%%%%%

\section{Small minimal counterexamples}

%To establish the inductive base case for $n\leq 20$ and finish the proof, Kleitman and Pinchasi apply computer based linear programming methods along with Lemma~\ref{KPlem}.

We continue our search for minimal counterexamples with $n\leq 20$.
Similar to Kleitman and Pinchasi, our main tool is computer based linear programming.
We include as many extra constraints as we can to eliminate as many $n$ as possible.
In the end we are left with just two cases where a minimal counterexample may exist.
% set of potential $s_{i,j}$ for $n=6$.
We will eliminate these possibilities with direct geometric arguments.

As well as constraints arising from the previous discussion, we use Hirzebruch's Inequality~\cite{Hirzebruch}. As before, $t_i$ is the number of lines containing $i$ points in $P$.
Note that Corollary~\ref{easycoro} ensures that at most $|P|-3$ points are collinear.
\begin{theorem}[Hirzebruch's Inequality]\label{hirzebruch}
Let $P$ be a set of points with at most $|P|-3$ collinear. %Let $s_i$ be the number of lines containing exactly $i$ points of $P$. 
Then $$ t_2 +\frac{3}{4}t_3 \geq n + \sum_{i\geq5}(2i-9)t_i \enspace  . $$
\end{theorem}

We also introduce the following three constraints. % that is tight for general position colour classes.
\begin{observation}\label{newgpobs}
Suppose there are $n$ blue points. Each red point can lie on at most $\floor{n/2}$ lines determined by two or more blue points.
\end{observation}

%This lemma is also useful.
\begin{lemma}\label{tricky}
Suppose $P$ is a set of $n$ red and $n$ (or $n-1$) blue points, and suppose there is a line $L$ with $r$ red and  $b$ blue points. Let $r'= n-r$ and $b'= n-b$ (or $b'= n-1-b$). Then the number of bichromatic lines is at least $$ \min_{i\in [b']} \left\{ i + (r-1)\max \left\{ \ceil{\frac{b'}{i}} , i \right\} \right\} + 
\min_{i\in [r']} \left\{ i + (b-1)\max \left\{ \ceil{\frac{r'}{i}} , i \right\} \right\} \enspace .$$
Moreover, if $r,b\geq 1$ we may add one more to this total.
\end{lemma}
\begin{proof}
Consider the red points on $L$, and suppose that of them $r_1$ is contained in the least bichromatic lines, and the number of these lines (excluding $L$) is $i$. Then the other $r-1$ red points on $L$ are each contained in at least $i$ bichromatic lines (excluding $L$). But they are also contained in at least $\ceil{\frac{b'}{i}}$ bichromatic lines since some line through $r_1$ contains at least this many blue points.
\end{proof}

%And we note the following.
\begin{observation}\label{obs1off}
A minimal counterexample to Theorem~\ref{mainthm} must determine precisely $|P|-2$ bichromatic lines. If it has fewer we can remove any point to obtain a smaller counterexample.
\end{observation}

%%%%%%%%%%%%%%%%%%%%5

All in all the constraints are as follows.
For brevity they are stated only for the case of $n$ red and $n$ blue points. The case of $n-1$ blue points is very similar.

\begin{itemize}
\item $\sum \binom{i}{2} s_{i,j} = \binom{n}{2} $ (Counting red pairs)
\item $\sum \binom{j}{2} s_{i,j} = \binom{n}{2} $ (Counting blue pairs)
\item $\sum ij s_{i,j} = n^2 $ (Counting bichromatic pairs)
\item $\sum (i+j-3) s_{i,j} \leq -3 $ (Melchior's Inequality~\ref{melchior})
\item $s_{1,1} + s_{0,2} + s_{2,0} + \frac{3}{4}(s_{0,3}+ s_{1,2}+s_{2,1}+s_{3,0}) \geq 2n + \sum_{i+j\geq 5} (2i+2j-9)s_{i,j}$ (Hirzebruch's Inequality~\ref{hirzebruch})
\item If $i+j \geq n$ then $s_{i,j} =0$ (Lemma~\ref{KPlem})
\item If $i=0$ or $j=0$ then $s_{i,j} =0$ (Lemma~\ref{1redlem})
%\end{itemize}
%\begin{itemize}
\item If $i^2 +i +j^2 +j \geq 4n-2 $ then $s_{i,j} =0$ (Lemmas~\ref{imp21} and~\ref{KPlem})
\item $\sum_{j\geq 2} i s_{i,j} \leq n\floor{n/2} $ (Observation~\ref{newgpobs})
\item $\sum_{i\geq 2} j s_{i,j} \leq n\floor{n/2} $ (Observation~\ref{newgpobs})
\item Constraints from Lemma~\ref{tricky}
\item $\sum_{i,j\geq 1} s_{i,j} = 2n-2$ (Observation~\ref{obs1off})
\item $s_{i,j} \in \N_0$
\end{itemize}

%Table~\ref{tabber} shows the difference between the minimum number of bichromatic lines as given by this linear program\footnote{The input files for the linear programming software, as well as a program used to generate the files are available from the author's web page \url{www.ms.unimelb.edu.au/~mspayne/}.}, and the target bound of $|P|-1$. 
%Results are given for both the case of $n$ red and $n$ blue (even) and the case of $n$ red and $n-1$ blue (odd).
%A non-negative value in Table~\ref{tabber} indicates that there are no minimal counterexamples to Theorem~\ref{mainthm} for that case.

%, so in particular it is true for each case with $n \geq 10$. This can be combined with the inductive step of Kleitman and Pinchasi.

%\begin{theorem}\label{n10KPconj}
%Let $P$ be a set of $n$ red, and $n$ or $n-1$ blue points in the plane, where $n\geq 10$. If neither colour class is collinear, then $P$ determines at least $|P|-1$ bichromatic lines.
%\end{theorem}

Running this linear program\footnote{%
%The input files for the linear programming software, as well as a program used to generate the files are 
The program used to generate the linear programs for each case is
available from the author's web page \url{www.ms.unimelb.edu.au/~mspayne/}.}
for each case with $n\leq20$ yields just two cases with a feasible solution.
They are the cases of $8$ red and $7$ blue points, and $6$ red and $5$ blue points.

In the first case, with $8$ red and $7$ blue points, the linear program returns a solution with $s_{2,3}=3$. If one adds the constraint that $s_{2,3} =0$, there is no longer a feasible solution.
So suppose that $s_{2,3} \geq 1$.
Consider a line $L$ containing $3$ blue points $b_1,b_2$ and $b_3$, and not containing $6$ of the red points.
Using the proof method from Lemma~\ref{tricky}, one can check that $b_1$ has $2$ lines through the reds, and $b_2$ and $b_3$ have $3$ (other cases don't yield a counterexample).
Transform $L$ to the line at infinity with a projective transformation.
Then the $6$ red points lie on two parallel lines through $b_1$, with three reds on each.
They also lie on three parallel lines through $b_2$.
Finally, they should also lie on another set of three parallel lines through $b_3$.
This is clearly impossible -- for example, note that there is only one non-crossing straight edged matching on the six red points.
%See Figure ADD ME. PLEEEEEEEEEEEEEEEEEEEEEEEAAAAAAAAAAAAAAAAAAAAASSSSSSSSSSSSSEEEEEE!!!!!!!!!!!!!!!!!!!

In the case with $6$ red and $5$ blue points, the linear program returns a solution with $s_{2,2}=6$, $s_{0,2}=4$, $s_{2,0}=6$ and $s_{2,1}=3$.
If one adds the constraint that $s_{2,2}\leq5$, there is no longer a feasible solution.
Similarly, there are no solutions with $s_{2,2}\geq7$, and also none with $s_{2,1}\leq 2$.
%So suppose that $s_{2,2}= 6$ and $s_{2,1}=3$.
We will show that this is not geometrically realisable.
We will work in the projective plane and make use of the following well known fact. 
It is simply the statement that one projective basis can be transformed to another.

\begin{proposition}\label{projsquare}
Let $V$ and $W$ be real projective planes. Given $v_1, \ldots, v_4 \in V$ in general position and $w_1, \ldots , w_4 \in W$ in general position, there exists a unique collineation (a bijection that preserves collinearities) from $V$ to $W$ that maps each $v_i$ to $w_i$.
\end{proposition}

\begin{proposition}\label{66prop}
It is not possible to arrange $6$ red points and $5$ blue points in the plane so that $s_{2,2}=6$ and $s_{2,1}=3$.
\end{proposition}

\begin{proof}
Suppose for contradiction that $s_{2,2}=6$ and $s_{2,1}=3$.
This gives $30$ bichromatic pairs, so there can be no more bichromatic lines.
This implies that every blue point is on three lines containing two red points. %, and thus induces a matching of the six points of the opposite colour.
Label the points $r_1, \ldots, r_6$ and $b_1,\ldots, b_5$.
Suppose $\{ r_5,r_6,b_1, b_2\}$ lie on a line $L$.
Since $b_1$ is collinear with two pairs in $\{ r_1, r_2, r_3, r_4\}$, this set is in general position.
Hence by Proposition~\ref{projsquare} we may assume that they are the vertices of a square, with coordinates $(-1,1),(1,1), (-1,-1)$ and $(1,-1)$ respectively, as shown in Figure~\ref{66pic}.
Since $b_2$ is also collinear with two red pairs in $\{ r_1, r_2, r_3, r_4 \} $, we may also assume, without loss of generality, that\footnote{This is the point at infinity in the direction of the $x$-axis.}  $b_1 = \overline{r_1r_2} \cap \overline{r_3r_4}= (\infty,0)$ and $b_2 = \overline{r_1r_4}\cap \overline{r_2r_3}= (0,0)$.
%Thus $r_1$ is the point at infinity on $H$ and $r_2$ is the midpoint of the square $b_1b_2b_3b_4$. %=\overline{b_1b_4}\cap \overline{b_2b_3}$.

\begin{figure}%[!h]
\bc
\includegraphics{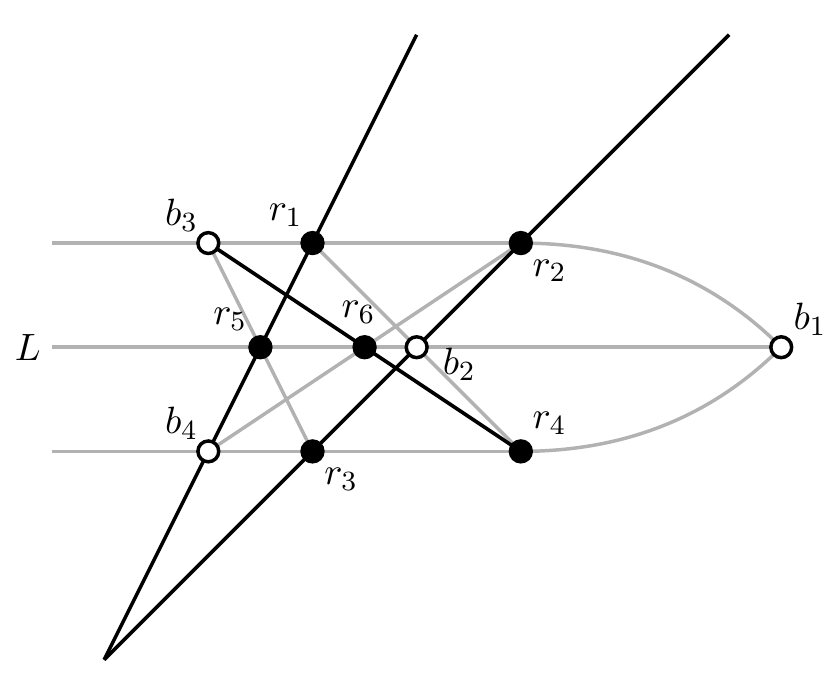}
\ec
\caption{Construction for Proposition~\ref{66prop}.}
\label{66pic}
\end{figure}

There is another blue point on the line $\overline{r_1r_2}$ (with equation $y=1$), say $b_3$, and a further blue point on $\overline{r_3r_4}$ (with equation $y=-1$), say $b_4$.
The position of either $b_3$ or $b_4$ determines the set $\{r_5,r_6\}$.
That is, $\{r_5,r_6\}=\{L \cap \overline{b_3r_3},L \cap \overline{b_3r_4}\}=\{L \cap \overline{b_4r_1}, L \cap \overline{b_4r_2} \}$.
%It can be seen from the symmetry of Figure~\ref{66pic} that
Since the configuration described thus far is symmetric about the line $y=0$, it follows that 
if $b_3=(a,1)$ for some real number $a$, then $b_4=(a,-1)$.

At this stage there are six bichromatic lines with only one blue point: $\overline{r_1r_4},$ $\overline{r_4r_6}, \overline{r_6r_2}, \overline{r_2r_3}, \overline{r_3r_5}$ and $\overline{r_5r_1}$.
There is one blue point $b_5$ left to determine, and it must lie on three of these lines.
Note that the bichromatic lines form a cycle on the blue points in the order listed.
Neighbours in the cycle share a red point, so cannot share a blue point, and so $b_5$ lies on alternating lines in the cycle.
By symmetry in the line $y=0$, we may assume $b_5$ lies on $\overline{r_2r_3},\overline{r_4r_6}$ and $\overline{r_5r_1}$.

Since $\overline{r_2r_3}$ is the line $x=y$, we can say that $b_5=(c,c)$ for some real number\footnote{The point $b_5$ could also be at infinity on $\overline{r_2r_3}$. This case is easily excluded by inspection since both $\overline{r_4r_6}$ and $\overline{r_5r_1}$ would need to be parallel to $\overline{r_2r_3}$. There is no value of $a$ that achieves this.} $c$.
Since $b_5$ lies on $\overline{r_5r_1}=\overline{b_4r_1}$, we have $$ (c,c) = \lambda(a,-1) + (1-\lambda)(-1,1) $$ for some parameter $\lambda$.
Eliminating $\lambda$ from these two equations yields $$ ac = a-1-3c \enspace .$$ 
Similarly, since $b_5$ lies on $\overline{r_4r_6}=\overline{b_3r_4}$, we have $$ (c,c) = \gamma(a,1) + (1-\gamma)(1,-1) $$ for some parameter $\gamma$.
Eliminating $\gamma$ from these two equations yields $$ac=3c-a-1 \enspace .$$
Equating both expressions for $ac$ yields $a =3c$, and substituting this into the above equation yields $3c^2 = -1$. This contradiction concludes the proof. 
\end{proof}

%Proposition~\ref{66prop} implies that $s_{2,2}\leq 8$ in the case of six red and six blue points.
%Adding this as an extra constraint in our linear program results in the minimum number of bichromatic lines increasing to 10.

%Thus we are one step closer to the complete Kleitman--Pinchasi conjecture for all $n$.

%\begin{theorem}\label{-2KPconj}
%Let $P$ be a set of $n$ red, and $n$ or $n-1$ blue points in the plane. If neither colour class is collinear, then $P$ determines at least $|P|-2$ bichromatic lines.
%\end{theorem}

%% don't mention comb setting too much

%%Finally, we note that without the geometric restrictions, the system of bichromatic lines in the proof of Proposition~\ref{66prop} can be completed as a combinatorial structure.
%%In addition to the lines already constructed (as shown in Figure~\ref{66pic})  we may include $r_5$ in the quadruples $\{r_2,r_5,b_2,b_3\},\{r_3,r_5,b_4, \allowbreak b_6\}$ and $\{r_1,r_5,b_1,b_5\}$ and $r_6$ in the quadruples $\{r_4,r_6,b_2,b_6\}, \{r_3,r_6,b_3,b_5\}$ and $\{r_2,\allowbreak r_6, b_1, b_4\}$.
%%This means that the Kleitman-Pinchasi Conjecture is not true in a combinatorial setting such as that of Theorem~\ref{meshu} of Meshulam~\cite{meshulam}.
%%It is an interesting question whether this combinatorial construction can be generalised to an infinite family of counterexamples.

\subsection*{Acknowledgements} I would like to thank Brendan McKay for some fruitful discussions. % and david???

\bibliography{minbib}

\bibliographystyle{plain} %myNatbibStyle}

%%%%%%%%%%%%%%%%%%%%%%%%
\end{document}